\documentclass[12pt,reqno]{amsart}

\usepackage{amssymb, xcolor}
\usepackage{amscd}
\usepackage{amsfonts}
\usepackage{setspace}
\usepackage{version}
\usepackage{tikz}
\usepackage{caption}
\usepackage{hyperref}
\usepackage{comment}
\usepackage{dsfont}

\usepackage{graphicx}

\newtheorem{theorem}{Theorem}[section]
\newtheorem{proposition}[theorem]{Proposition}
\newtheorem{remark}[theorem]{Remark}
\newtheorem{lemma}[theorem]{Lemma}
\newtheorem{corollary}[theorem]{Corollary}

\newtheorem{question}[theorem]{Question}

\hypersetup{
    colorlinks=true,
    linkcolor=blue,
    filecolor=magenta,      
    urlcolor=cyan,
    pdftitle={Overleaf Example},
    pdfpagemode=FullScreen,
    }

\theoremstyle{definition}

\renewcommand{\leq}{\leqslant}
\renewcommand{\geq}{\geqslant}

\def\R{\mathbb{R}}

\def\Z{\mathbb{Z}}

\numberwithin{equation}{section}
\usepackage{amsmath, amssymb, amsthm, mathtools}
\usepackage[a4paper,margin=1in]{geometry}

\usepackage{color}
\usepackage{amssymb}

\newcommand{\newr}[1]{\textcolor{red}{#1}}

\begin{document}

\title[On the frequency function of Hardy-Littlewood maximal functions]{On the frequency function of Hardy-Littlewood maximal functions}

\author[C. Garz\'on]{Carlos Garz\'on}
\address[CG]{Department of Mathematics, Virginia Polytechnic Institute and State University,  225 Stanger Street, Blacksburg, VA 24061-1026, USA
}
\email{cgarzongu@vt.edu}

\author[J. Madrid]{Jos\'e Madrid}
\address[JM]{Department of Mathematics, Virginia Polytechnic Institute and State University,  225 Stanger Street, Blacksburg, VA 24061-1026, USA
}
\email{josemadrid@vt.edu}


\date{}

\begin{abstract}
    We study the frequency function (introduced by Temur in \cite{TemurDiscrete}) in both the discrete and continuous settings. More precisely, we extend the definition of the frequency function to the higher-dimensional continuous setting and to the uncentered Hardy-Littlewood maximal function. We analyze the asymptotic behavior of the frequency function and the density of its small values for functions in $\ell^1(\mathbb{Z)}$ and $L^1(\mathbb{R}^d)$ answering some questions posed by Temur in \cite{TemurDiscrete}. Finally, we study the size of the frequency function for functions in $\ell^p(\mathbb{Z})$ with $p>1$, showing that this case differs significantly from the case $p=1$.
\end{abstract}

\maketitle

\section{Introduction}

Let $d$ be a positive integer, and let $f \in L^1_{\text{loc}}(\mathbb{R}^d)$. We define the average of $f$ over the ball centered at $x \in \mathbb{R}^d$ of radius $r>0$ by
\begin{equation*}
    A_rf(x) = \frac{1}{\mu(B_r(x))} \int_{B_r(x)} f(y) \, dy,
\end{equation*}
where $\mu$ denotes the Lebesgue measure in $\mathbb{R}^d$ and $B_r(x)$ stands for the open ball centered at $x$ with radius $r$. Additionally, we define the centered Hardy-Littlewood maximal function of $f$ by 
\begin{equation} \label{DefCentMax}
    Mf(x) = \sup_{r>0} A_r|f|(x) = \sup_{r>0} \frac{1}{\mu(B_r(x))} \int_{B_r(x)} |f(y)| \, dy.
\end{equation}
A classical argument involving Vitali's covering lemma shows that $M$ is a bounded operator from $L^{1}(\mathbb{R}^d)$ to $L^{1,\infty}(\mathbb{R}^d)$, and an application of Marcinkiewicz interpolation theorem (combined with the trivial inequality $\|Mf\|_{L^{\infty}(\mathbb{R}^d)} \leq \|f\|_{L^{\infty}(\mathbb{R}^d)}$) establishes the boundedness of $M$ on $L^p(\mathbb{R}^d)$ for all $p \in (1,+\infty]$. Kinnunen started the study of the regularity properties of the Hardy-Littlewood maximal operator on Sobolev spaces in \cite{Kinnunen1997}, showing that $M$ is bounded on $W^{1,p}(\mathbb{R}^d)$ for every $p \in (1,+\infty]$. This result encouraged further research on the regularity of both the centered and uncentered (defined in \eqref{uncenteredHL}) Hardy-Littlewood maximal operators in the continuous and discrete settings, see \cite{HajlaszOnninen2004,Tanaka2002,AldazPerez2007,Kurka2015,BoberCarneiroHughesPierce2012,CarneiroHughes2012,Luiro2018} and references therein. It is worth pointing out that Luiro \cite{Luiro2007} proved the continuity of $M$ in Sobolev spaces $W^{1,p}(\mathbb{R}^d)$ for all $p \in (1,+\infty)$. In his argument, an analysis of the set of radii
\begin{equation}\label{Efx}
    E_{f,x} := \{r>0 : Mf(x) = A_r |f|(x)\}
\end{equation}
played an important role. See also \cite[Lemma 2.10]{LM2017}, \cite[Lemma 2.1]{BM2019} and \cite[Proposition 3.1]{BM2019.2}. Furthermore, Kurka \cite{Kurka2015} used these radii to examine the local maxima of $Mf$, which allowed him to show the existence of a constant $C$ such that
\begin{equation} \label{VarMf}
    \text{Var}(Mf) \leq C ~ \text{Var}(f)
\end{equation}
for all functions $f : \mathbb{R} \to \mathbb{R}$ of bounded variation. An analogous idea was also used by Temur \cite{Temur2017} to establish \eqref{VarMf} in the discrete setting. This motivated Temur \cite{TemurDiscrete,TemurContinuous} to study the size and smoothness properties of the function $x \mapsto \inf E_{f,x}$ in the one-dimensional discrete and continuous settings. This function is well-defined for functions in $\ell^1(\mathbb{Z})$, while its definition is more subtle for functions in $L^1(\mathbb{R})$. An interesting application of this function is addressed in \cite{Steinerberger2015}, where Steinerberger used it to give a characterization of the sine function. We proceed to introduce the accurate definition of this function (which we call \textit{frequency function}) in each setting and we discuss some of the most important properties known about it.

Let $f \in \ell^1(\mathbb{Z})$. We define the average of $f$ over the interval of radius $r \in \mathbb{N}$ centered at $n \in \mathbb{Z}$ by
\begin{equation*}
    A_rf(n) = \frac{1}{2r+1} \sum_{j=-r}^{r} f(n+j),
\end{equation*}
and we define the discrete centered Hardy-Littlewood maximal function of $f$ by
\begin{equation*}
    Mf(n) = \sup_{r \in \mathbb{N}} A_r|f|(n) = \sup_{r \in \mathbb{N}} \frac{1}{2r+1} \sum_{j=-r}^{r} |f(n+j)|,
\end{equation*}
where $\mathbb{N}$ stands for the set of non-negative integers. In \cite[Section 2]{TemurDiscrete}, it is observed that for each $n\in \mathbb{Z}$ there exists $\rho_n \in \mathbb{N}$ such that $Mf(n) = A_{\rho_n}|f|(n)$. Hence, the set of good radii
\begin{equation*}
    E_{f,n}:=\{r \in \mathbb{N}:  Mf(n)=A_r |f|(n)\}
\end{equation*}
is non-empty. The discrete frequency function of $f$ is defined as 
\begin{equation*}
    r_n := \min E_{f,n} = \min\{r \in \mathbb{N}:\ Mf(n)=A_r |f|(n)\}.
\end{equation*}
Since the mass of a summable function is concentrated on a finite interval centered at the origin, it is expected that $r_n \sim |n|$ as $n \to +\infty$ (where the notation $g(n) \sim h(n)$ as $n \to +\infty$ means that $\lim_{n\to \infty} g(n)/h(n) = 1$). This motivated Temur to prove that the set 
\begin{equation} \label{ScSets}
    S_C = \left\{ n \in \mathbb{Z} : \frac{1}{2C} \leq \frac{r_n}{|n|} \leq \frac{1}{C} \right\}
\end{equation}
is finite for each $C>1$ \cite[Theorem 1]{TemurDiscrete}. Our first main theorem in this paper provides a generalization of this result. More precisely, in Theorem \ref{thm: Thm1}, we prove that $r_n/|n|$ is arbitrarily close to $1$ or $0$ outside of a finite interval centered at the origin. This result implies \cite[Theorem 1]{TemurDiscrete} and also that the set $S_C$ is finite for all $C\in (0,1/2)$. In the case $1/2 \leq C \leq 1$, it is not possible to guarantee the finiteness of the set $S_C$. For instance, for the Dirac delta function
\begin{equation} \label{DeltaDirac}
    f(n) = \begin{cases}
        1,& \text{if } n = 0,\\
        0, & \text{otherwise}
    \end{cases}
\end{equation}
we have that $r_n = |n|$ for every integer $n$, which means that $S_C$ is infinite for each $1/2\leq C\leq 1$.

Moreover, we may construct examples of summable functions such that $r_n = 0$ for infinitely many integers (such as the function from Section \ref{Sec:Proof2}). This inspired Temur to study the density of the set of integers where the frequency function vanishes or takes small values. For a fixed $C>1$, we define 
\begin{equation*}
    S_{N,f}:=\left\{\,n\in\mathbb Z:\ |n|\le N,\  \frac{r_n}{|n|} \leq \frac{1}{C}\,\right\}.
\end{equation*}
Temur proved that 
\begin{equation}\label{DensityRn}
    \lim_{N\to \infty} \frac{|S_{N,f}|}{N} = 0
\end{equation}
for every $C>1$ \cite[Theorem 2]{TemurDiscrete}. He also showed that this result no longer holds if the denominator is replaced by $N^{1-\varepsilon}$ or $N/\log^{1+\varepsilon}(N)$ for any positive constant $\varepsilon$. Moreover, he asked the following question.

\begin{question}{\cite[Open Problem 2]{TemurDiscrete}} \label{QuestionNlogN}
Let $f \in \ell^1(\mathbb{Z})$ be a function that is not identically zero. Let $C>1$ be a real number. Is the following statement true?
\begin{equation} \label{LimQuestion}
    \lim_{N\to \infty} \frac{|S_{N,f}|}{N/\log(N)} = 0.
\end{equation}
\end{question}

In Theorem \ref{thm: counterexample g}, we exhibit an example to prove that the answer to this question is negative. In fact, we prove that \eqref{DensityRn} does not hold if the denominator is $N/g(N)$ for an increasing function $g$ tending to $\infty$ as $N\to\infty$, which provides a more general result than the original question. 

It is also natural to wonder what happens with the sets $S_C$ if the function $f$ belongs to $\ell^p(\mathbb{Z})$ for $p>1$. Using H\"older's inequality and the same idea as in the $\ell^1$-case, we observe that the discrete frequency function is well-defined for all $p \in (1,+\infty)$. Nevertheless, for any $1<p<+\infty$, we can construct functions in $\ell^p(\mathbb{Z})$ such that $r_n/|n| \sim 1/4$ for infinitely many $n$'s, which means that the sets $S_C$ are not necessarily finite in the $\ell^p$-case. This construction will be carried out in Theorem \ref{thrm:lp}. It is worth mentioning that the discrete frequency function is also well-defined for functions in $\ell^{\infty}(\mathbb{Z})$ except (perhaps) on those points where $Mf$ attains global minima, see \cite[Lemma 3]{BoberCarneiroHughesPierce2012}.

Another open question posed in \cite{TemurDiscrete} is the construction of the frequency function for the discrete uncentered Hardy-Littlewood maximal operator \cite[Open Problem 5]{TemurDiscrete}. Let $f \in \ell^1(\mathbb{Z})$. For any $\rho,s \in \mathbb{N}$, we define the average of $f$ over an interval of length $\rho+s+1$ containing $n \in \mathbb{Z}$ by
\begin{equation*}
    A_{\rho,s}f(n) = \frac{1}{\rho+s+1} \sum_{j=-\rho}^{s} f(n+j),
\end{equation*}
and we define the discrete uncentered Hardy-Littlewood maximal function of $f$ by
\begin{equation} 
    \widetilde{M}f(n) = \sup_{\rho,s \in \mathbb{N}} A_{\rho,s}|f|(n) = \sup_{\rho,s \in \mathbb{N}} \frac{1}{\rho+s+1} \sum_{j=-\rho}^{s} |f(n+j)|.
\end{equation}
By using a slight variation of the argument for the centered case, we may show that for each $n \in \mathbb{Z}$ there exist $\rho_n,s_n \in \mathbb{N}$ such that $\widetilde{M}f(n) = A_{\rho_n,s_n}|f|(n)$. As a result, the set
\begin{equation*}
    \widetilde{E}_{f,n} = \{\rho+s :\widetilde Mf(n) = A_{\rho,s}|f|(n)\}
\end{equation*}
is non-empty. We define the discrete frequency function of $f$ in the uncentered case by
\begin{equation*}
    \widetilde{r}_n = \frac{1}{2} \min \widetilde{E}_{f,n} = \frac{1}{2} \min\{\rho+s : \widetilde Mf(n) = A_{\rho,s}|f|(n)\}.
\end{equation*}
By using again the fact that the mass of a summable function is concentrated on a finite interval centered at the origin, it is expected that $\widetilde{r}_n \sim |n|/2$ when $n$ is large enough. In fact, this is the case for the Dirac delta function defined in \eqref{DeltaDirac}. By applying the same arguments from Theorem \ref{thm: Thm1}, we may prove that $\widetilde{r}_n/|n|$ is arbitrarily close to $1/2$ or $0$ outside of a finite interval centered at the origin.

We now turn our attention back to the continuous setting. Let $f \in L^1(\mathbb{R}^d)$. Recalling the definition of the sets $E_{f,x}$ given in \eqref{Efx}, we define the frequency function of $f$ by
\begin{equation*}
    r_x = \begin{cases}
        \inf E_{f,x} , &\text{ if }E_{f,x} \neq \emptyset,\\
        0, &\text{otherwise}.
    \end{cases}
\end{equation*}
This function is clearly well-defined. In \cite[Propositions 2.2 and 2.3]{TemurContinuous}, it was shown that if $r_x>0$ then $Mf(x) = A_{r_x}|f|(x)$, while if $r_x=0$, there exists a sequence of positive numbers $(\rho_n)$ such that $A_{\rho_n}|f|(x) \to Mf(x)$ and $\rho_n\to 0$ as $n\to +\infty$. It was also established the measurability of the function $x \mapsto r_x$ \cite[Theorem 1.1]{TemurContinuous}. The proofs of these facts were carried out in the one-dimensional case, but the same arguments work for the higher-dimensional case. These properties allow us to deal with the frequency function in the continuous setting very similarly to the discrete setting, and analogous ideas can be used to deduce that the sets $S_C$ (defined in \eqref{ScSets}) are bounded for all $C>1$ \cite[Theorem 1.2]{TemurContinuous} and that the limit \eqref{DensityRn} holds in the one-dimensional continuous setting \cite[Theorem 1.3]{TemurContinuous}. We prove that both results also hold in the higher-dimensional case, as we shall see in Theorem \ref{thm: Thm1} and Proposition \ref{prop:Prop1}.

Finally, we construct the frequency function in the uncentered continuous case. Let $f \in L^1(\mathbb{R}^d)$. We define the uncentered Hardy-Littlewood maximal function of $f$ by
\begin{equation} \label{uncenteredHL}
    \widetilde{M}f(x) = \sup_{\overline{B_r(y)} \ni x} A_r |f|(y) = \sup_{\overline{B_r(y)} \ni x} \frac{1}{\mu(B_{r}(y))} \int_{B_r(y)} |f(t)| \, dt,
\end{equation}
where $\overline{B_r(y)}$ stands for the closure of the open ball $B_r(y)$ \footnote{Observe that the definition of the centered maximal function \eqref{DefCentMax} remains the same if we consider closed balls instead of open balls, as we did in \eqref{uncenteredHL}.}. For each $x \in \mathbb{R}^d$, we define the set of good radii as
\begin{equation*}
    \widetilde{E}_{f,x} := \left\{ r>0 : \text{there is } y \in \mathbb{R}^d \text{ such that } x \in \overline{B_r(y)} \text{ and } \widetilde{M}f(x) =  A_r|f|(y) \right\},
\end{equation*}
and we define the frequency function of $f$ by
\begin{equation*}
    \widetilde{r}_x = \begin{cases}
        \inf \widetilde{E}_{f,x}, &\text{if } \widetilde{E}_{f,x} \neq \emptyset,\\ 
        0, &\text{otherwise}.
    \end{cases}
\end{equation*}
Applying a similar argument to that of the centered case \cite[Propositions 2.2 and 2.3]{TemurContinuous}, we may deduce the following property: If $\widetilde{r}_x > 0$, there exists $y \in \mathbb{R}^d$ such that $ x \in \overline{B_{\widetilde{r}_x}(y)}$ and $\widetilde{M}f(x) = A_{\widetilde{r}_x}|f|(y)$. If $\widetilde{r}_x = 0$, there exist a sequence $(y_n)$ in $\mathbb{R}^d$ and a sequence $(\rho_n)$ in $\mathbb{R}^+$ such that $ x \in \overline{B_{\rho_n}(y_n)}$ for all $n \in \mathbb{Z}^+$, $\rho_n \to 0$ and $A_{\rho_n}|f|(y_n) \to \widetilde{M}f(x)$ as $n \to +\infty$. The proof of this property relies on the continuity of the function $(y,r) \in \mathbb{R}^d \times \mathbb{R}^+ \mapsto A_rf(y)$ (see \cite[Lemma 3.16]{FollandRealBook}) and the Bolzano-Weierstrass theorem, and it is useful to deal with the frequency function in the uncentered context very similarly to the centered case. Additionally, adapting the argument of \cite[Theorem 1.1]{TemurContinuous} with some minor modifications, we can deduce the measurability of the function $x \mapsto \widetilde{r}_x$. For the sake of completeness, we include this proof in Appendix \ref{AppendixProof}. Therefore, the main results of this paper have a natural extension to the uncentered continuous setting. 

\subsection{Organization of the paper} In Section \ref{MainResults} we state the main results of this paper. The other sections are devoted to present the proofs of these results. In Section \ref{Sec:Proof1} we prove Theorem \ref{thm: Thm1}, which generalizes \cite[Theorem 1]{TemurDiscrete} and \cite[Theorem 1.2]{TemurContinuous}. In Section \ref{Sec:Proof2} we show Theorem \ref{thm: counterexample g}, which answers Question \ref{QuestionNlogN}. In Section \ref{Sec:Proof3} we present a proof of Theorem \ref{thrm:lp}, establishing that Theorem \ref{thm: Thm1} does not hold for functions in $\ell^p(\mathbb{Z})$ with $p>1$. Finally, we show the measurability of the function $\widetilde{r}_x$ and sketch the proof of Proposition \ref{prop:Prop1} in Appendix \ref{AppendixProof}.

\section{Main results} \label{MainResults}

We denote by $(X,\mu)$ or $(X,\nu)$ the Euclidean space $\R^d$ with the Lebesgue measure or the set of the integers $\Z$ with the counting measure, respectively. In addition, we denote by $C_d:=\mu(B_1(0))$ the volume of the unit ball in $\R^d$.


\begin{theorem} \label{thm: Thm1}
Let $f\in L^{1}(X)$. For all $\varepsilon>0$, the set
\begin{equation} \label{ThrmRn/n}
\left\{x\in X : \frac{r_x}{|x|}\notin[0,\varepsilon)\cup(1-\varepsilon,1+\varepsilon)\right\}   
\end{equation}
is bounded.
\end{theorem}

\begin{remark} \label{rem: Rem1}
    The proof of Theorem \ref{thm: Thm1}, with minor modifications, allows us to obtain an analogous result for the uncentered Hardy-Littlewood maximal function: Let $f\in L^{1}(X)$. For all $\varepsilon>0$, the set
\begin{equation*}
\left\{x\in X : \frac{\widetilde r_x}{|x|}\notin[0,\varepsilon)\cup\left(\frac{1}{2} - \varepsilon , \frac{1}{2} + \varepsilon\right)\right\}   
\end{equation*}
is bounded. 
\end{remark}

Similarly to \eqref{DensityRn} we have the next result.

\begin{proposition} \label{prop:Prop1}
Let $0\neq f \in L^1(\mathbb{R}^d)$. For all $C>1$, the following identity holds
\begin{equation} \label{eq:Prop1}
    \lim_{N\to\infty}\frac{\mu(S_{N,f})}{N^d}=0,
\end{equation}
where 
\begin{equation*}
    S_{N,f}:=\left\{\,x\in\mathbb{R}^d:\ |x|\le N,\  \frac{r_x}{|x|} \leq \frac{1}{C}\,\right\}.
\end{equation*}
Similarly, let $0\neq f \in L^1(X)$. For all $C>2$, the following identities hold
\begin{equation} \label{LimUncentered}
    \lim_{N\to\infty}\frac{\mu(\widetilde{S}_{N,f})}{N^d}=0 \quad \text{and} \quad \lim_{N\to\infty}\frac{\nu(\widetilde{S}_{N,f})}{N}=0,
\end{equation}
where 
\begin{equation*}
    \widetilde{S}_{N,f}:=\left\{\,x\in X:\ |x|\le N,\  \frac{\widetilde{r}_x}{|x|} \leq \frac{1}{C}\,\right\}.
\end{equation*}
\end{proposition}

\begin{remark}
    The proof of this proposition follows very similar arguments to those of \cite[Theorem 1.3]{TemurContinuous} and \cite[Theorem 2]{TemurDiscrete}. For the sake of completeness, we include a sketch of this proof in Appendix \ref{AppendixProof}. 
\end{remark}

Combining Theorem \ref{thm: Thm1}, Proposition \ref{prop:Prop1}, and \eqref{DensityRn}, we obtain the following result.

\begin{corollary} \label{prop:DensityHigher}
Let $0\neq f\in L^{1}(X)$. For all $\varepsilon>0$, we have that
\begin{equation}\label{eq: lim=1}
\lim_{N\to\infty}\frac{\mu( \{\,x \in \mathbb{R}^d:\ |x|\le N,\ 1-\varepsilon\leq\frac{r_x}{|x|}\leq 1+\varepsilon\,\})}{C_d N^d}=1
\end{equation} 
and
\begin{equation}\label{eq: lim=2}
\lim_{N\to\infty}\frac{\nu( \{\,n \in \mathbb{Z}:\ |n|\le N,\ 1-\varepsilon\leq\frac{r_n}{|n|}\leq 1+\varepsilon\,\})}{2N}=1.
\end{equation}
\end{corollary}

\begin{remark}
    (i) Observe that \eqref{eq: lim=1} and \eqref{eq: lim=2} trivially hold for compactly supported functions.

    (ii) By combining Remark \ref{rem: Rem1} and Proposition \ref{prop:Prop1}, we deduce that a similar result holds for the uncentered Hardy-Littlewood maximal function: For all $0\neq f \in L^1(X)$ and all $\varepsilon>0$, we have
    \begin{equation*}
        \lim_{N\to\infty}\frac{\mu(\{\,x \in \mathbb{R}^d:\ |x|\le N,\ \frac{1}{2}-\varepsilon\leq\frac{\widetilde{r}_x}{|x|}\leq \frac{1}{2}+\varepsilon\,\})}{C_d N^d}=1
    \end{equation*}
    and
    \begin{equation*}
        \lim_{N\to\infty}\frac{\nu(\{\,n \in \mathbb{Z}:\ |n|\le N,\ \frac{1}{2}-\varepsilon\leq\frac{\widetilde{r}_n}{|n|}\leq \frac{1}{2}+\varepsilon\,\})}{2N}=1.
    \end{equation*}
\end{remark}

In the following theorem, we study the density of the set of zeros of the frequency function. A particular case of our result answers Question \ref{QuestionNlogN}, but this theorem is more general.

\begin{theorem}
\label{thm: counterexample g}
For each function $f \in L^1(X)$ and each $N \in \mathbb{Z}^+$, we define the set
\begin{equation*}
    Z_{N,f}:=\left\{\,x\in X:\ |x|\le N,\  r_x = 0\,\right\}.
\end{equation*}
Let $g:\mathbb N\to(0,\infty)$ be a non-decreasing function such that $g(N)\to\infty$ as $N\to\infty$.
Then:
\begin{itemize}
\item[(i)] There exists $0\neq f\in L^1(\R^d)$ such that
\[
\limsup_{N\to\infty}\frac{\mu(Z_{N,f})}{N^d/g(N)} \ge \frac{C_d}{2^d}.
\]
\item[(ii)] There exists $0\neq f\in \ell^1(\Z)$ such that
\[
\limsup_{N\to\infty}\frac{\nu(Z_{N,f})}{N/g(N)} \ge 1.
\]
\end{itemize}
\end{theorem}

\begin{remark}
    (i) Notice that the slower $g$ increases, the stronger the previous theorem is.

    (ii) As a consequence of this theorem, there are non-zero functions $f \in L^1(\mathbb{R}^d)$ and $h \in \ell^1(\mathbb{Z})$ such that 
    \begin{equation*}
        \limsup_{N\to\infty}\frac{\mu(S_{N,f})}{N^d/g(N)} \ge \frac{C_d}{2^d} \quad \text{and} \quad \limsup_{N\to\infty}\frac{\nu(S_{N,h})}{N/g(N)} \ge 1.
    \end{equation*}
    By considering $g(N) = \log(N)$, we deduce that the answer to Question \ref{QuestionNlogN} is negative. In fact, this theorem also shows that the value of the limit \eqref{LimQuestion} is different from $0$ if we change the logarithmic function in the denominator for some other functions with slower growth, such as $\log(\log(N))$ or $\log^{1-\varepsilon}(N)$ for $\varepsilon>0$.
\end{remark}

Finally, the statement in Theorem \ref{thm: Thm1} does not hold for $p>1$, that is the content of our next result.

\begin{theorem} \label{thrm:lp}
For each $p \in (1,+\infty]$ there exists a function $f \in \ell^{p}(\mathbb{Z})$ such that the set
    \begin{equation*}
        \left\{n\in \mathbb{Z} : \frac{r_n}{|n|}\in \left[\frac{1}{8} , \frac{7}{8}\right]\right\}
    \end{equation*}
    is unbounded.
\end{theorem}
\begin{remark}
    (i) In our proof for the case $p=+\infty$, we construct a function $f \in \ell^{\infty}(\mathbb{Z})$ such that the set 
    \begin{equation*}
        \left\{n\in \mathbb{Z} : \frac{r_n}{|n|}\in \left[\frac{1}{4} , \frac{3}{4}\right]\right\}
    \end{equation*}
    is unbounded, which is a better result than that of the case $p \in (1,+\infty)$.

    (ii) The same functions constructed in the proof of this theorem show that the analogous result for the uncentered Hardy-Littlewood maximal function holds: For each $p \in (1,+\infty]$ there exist $\varepsilon>0$ and a function $f \in \ell^{p}(\mathbb{Z})$ such that the set
    \begin{equation*}
        \left\{n\in \mathbb{Z} : \frac{\widetilde{r}_n}{|n|}\in \left[\frac{1}{16} , \frac{7}{16} \right]\right\}
    \end{equation*}
    is unbounded.

    (iii) For $p = +\infty$, we readily see that the sets 
    \begin{equation*}
        \left\{n\in \mathbb{Z} : \frac{r_n}{|n|} \geq 1 + \varepsilon\right\} \quad \text{and} \quad \left\{n\in \mathbb{Z} : \frac{\widetilde{r}_n}{|n|} \geq \frac{1}{2} + \varepsilon\right\}
    \end{equation*}
    are not necessarily bounded. Indeed, for the function $f = \mathds{1}_{(-\infty,0]} \in \ell^{\infty}(\mathbb{Z})$, we have that $r_n = \widetilde{r}_n = +\infty$ for all positive integers $n$. 
\end{remark}

\section{Asymptotic behavior of the frequency function}\label{Sec:Proof1}

\begin{proof}[Proof of Theorem \ref{thm: Thm1}]
We discuss the case when $X = \mathbb{R}^d$. The case $X = \Z$ follows similarly. If $f = 0$, the result follows immediately. Then we suppose that $f\neq 0$.

We proceed via contradiction, let us assume that the set 
\begin{equation*}
    \left\{ x \in X : \frac{r_x}{|x|} \in [\varepsilon,1-\varepsilon] \right\}
\end{equation*}
is not bounded. Hence, there exists a sequence $(x_k)$ in this set such that $|x_k| \to \infty$ as $k \to +\infty$. Without loss of generality, we may assume that $(2-\varepsilon)|x_k| < \varepsilon |x_{k+1}|$ for each $k \in \mathbb{Z}^+$. Then, the annuli
\begin{equation*}
    U_k = \{y \in X : \varepsilon|x_k| \leq |y| \leq (2-\varepsilon)|x_k|\}
\end{equation*}
are pairwise disjoint. Hence,
\begin{equation} \label{ContraConv}
    \sum_{k=1}^{\infty}\int_{U_k}|f(y)| \, dy \leq \int_X |f(y)|\, dy < \infty.
\end{equation}
Since $(1-\varepsilon)|x_k|\geq r_{x_k}\geq\varepsilon|x_k|>0$ for every positive integer $k$, as a consequence of the triangle inequality
\begin{equation} \label{UpperM}
    Mf(x_k) = A_{r_{x_k}}|f|(x_k) = \frac{1}{\mu(B_{r_{x_k}}(x_k))} \int_{B_{r_{x_k}}(x_k)} |f(y)|\, dy \leq \frac{1}{C_d \varepsilon^d |x_k|^d} \int_{U_k}|f(y)| \, dy
\end{equation}
for all $k\in\Z^+$. Let $R$ be a positive constant such that
\begin{equation*}
    \int_{B_R(0)}|f(y)|\, dy \geq \frac{\|f\|_1}{2}.
\end{equation*}
We assume without loss of generality that $|x_1| > R$. Thus,
\begin{equation} \label{LowerM}
    Mf(x_k) \geq A_{2|x_k|}|f|(x_k) \geq \frac{\|f\|_1}{C_d 2^{d+1} |x_k|^d}
\end{equation}
for all $k\in\Z^+$. From \eqref{UpperM} and \eqref{LowerM}, we obtain
\begin{equation*}
    \frac{\varepsilon^d}{2^{d+1}} \|f\|_1 \leq  \int_{U_k}|f(y)| \, dy
\end{equation*}
for all $k\in\Z^+$, which contradicts \eqref{ContraConv}. Therefore,
\begin{equation*}
    \left\{ x \in X : \frac{r_x}{|x|} \in [\varepsilon,1-\varepsilon] \right\}
\end{equation*}
is a bounded set. Then there is $R'>0$ such that
\begin{equation*}
    \left\{ x \in X : \frac{r_x}{|x|} \in [\varepsilon,1-\varepsilon] \right\} \subseteq B_{R'}(0).
\end{equation*}
Observe that for fixed $\varepsilon>0$ and $d\in\Z^+$, there are positive constants $K$ and $\delta$ such that
\begin{equation}\label{ineq: k, delta, epsilon, d}
    1 < K < (1-\delta)^{1/d}(1+\varepsilon).
\end{equation}
We assume without loss of generality that $R'$ is sufficiently large, such that 
\begin{equation}\label{ineq: 1-delta}
    \int_{B_{R'}(0)}|f(y)| \, dy \geq (1-\delta)\|f\|_1.
\end{equation}
We define $M:=\max\{R'/(K-1),R'\}$. By the triangle inequality, we have $B_{R'}(0) \subseteq B_{K|x|}(x)$ for all $x \in X \setminus B_M(0)$. Then, using \eqref{ineq: 1-delta} we obtain
\begin{equation*}
    Mf(x) \geq A_{K|x|}|f|(x) = \frac{1}{\mu(B_{K|x|}(x))} \int_{B_{K|x|}(x)} |f(y)| \, dy \geq \frac{1-\delta}{C_d K^d |x|^d} \|f\|_1,
\end{equation*}
for every $x \in X \setminus B_M(0)$. Let $x \in X\setminus B_{M}(0)$ be such that $\frac{r_x}{|x|} \not \in [0,\varepsilon)$. Notice that
\begin{equation*}
    Mf(x) = \frac{1}{\mu(B_{r_x}(x))} \int_{B_{r_x}(x)} |f(y)| \, dy \leq \frac{1}{C_d r_x^d} \|f\|_1.
\end{equation*}
Combining the two previous inequalities with \eqref{ineq: k, delta, epsilon, d}, we conclude 
\begin{equation*}
    \frac{r_x}{|x|} \leq \frac{K}{(1-\delta)^{1/d}} < 1+\varepsilon.
\end{equation*}
Therefore, the set \eqref{ThrmRn/n} is contained in the ball $B_M(0)$. This finishes the proof of the theorem.
\end{proof}

\section{Density of zeros of the frequency function} \label{Sec:Proof2}

\begin{proof}[Proof of Theorem \ref{thm: counterexample g}]
We discuss the case when $X = \mathbb{R}^d$. The case of $\Z$ with the counting measure follows essentially from the same proof, replacing $C_d$ by 2 and $d$ by 1.

Observe that if $N^d / g(N) \to 0$ as $N \to \infty$, the theorem follows immediately by considering the characteristic function $f:= \mathds{1}_{B_1(0)}$. Then, we assume without loss of generality that $N^d/g(N) \not\to 0$ as $N \to \infty$, which implies the existence of a positive integer $M>1$ and a strictly increasing sequence
$N_k\to\infty$, such that
\begin{align}
\frac{N_k}{g(N_k)^{\frac{1}{d}}} &\geq \frac{1}{M-1}\quad \text{for all } k\in\Z^+, \label{eq:nd/g-large}\\
N_{k+1} &\geq 10\, N_k \quad \text{for all } k\in\Z^+,\label{eq:sep}\\
N_1&\geq 4.\nonumber
\end{align}
Since $g(N)\to\infty$ as $N\to\infty$, we can also assume that
\begin{align}
g(N_k) &\ge M^d\left[\left[\frac{C_d}{2^k}\left(\frac{4}{5}\right)^d+1\right]^{\frac{1}{d}}-1\right]^{-d}\ \text{and}\ g(N_k)\geq 2^d \quad \text{for all } k\in\Z^+. \label{eq:g-large}
\end{align}
For each $k\in\Z^+$, we define
\[
L_k:=\left\lceil \frac{N_k}{g(N_k)^{\frac{1}{d}}}\right\rceil,\ \
\text{and}\ \ I_k:=\{x\in X : N_k < |x|< N_k+L_k, \ \text{and}\ \ x_i\geq 0 \ \text{for all}\ 1\leq i\leq d\} .
\]
Since $g(N_k)\geq 2^d$ for all $k\in\Z^+$, it follows that $L_k\leq \lceil N_k/2 \rceil < \lceil N_k-1 \rceil < N_k$ for each $k \in \mathbb{Z}^+$. Moreover, \eqref{eq:nd/g-large} implies that 
\begin{equation} \label{boundL/N}
    L_k \leq M \frac{N_k}{g(N_k)^{\frac{1}{d}}}
\end{equation}
for all $k \in \mathbb{Z}^+$. Let
\[
a_k:=\frac{1}{2^{k}\mu(I_k)}.
\]
Define $f:X\to[0,\infty)$ by
\[
f(x):=
\begin{cases}
a_k, & x\in I_k,\\
0, & \text{otherwise}.
\end{cases}
\]
Then
\[
\|f\|_{L^1(X)}=\sum_{k\ge 1} a_k\mu(I_k)
=\sum_{k\ge 1} \frac{1}{2^k}=1,
\]
in particular, $0\neq f\in L^1(X)$.

\begin{lemma}
Let $k\in\Z^+$. Then, $Mf(x)=a_k$ for all $x\in I_k$. In particular, $r_x = 0$ for every $x \in I_k$.    
\end{lemma}
\begin{proof}
Let $k$ be a positive integer and $x\in I_k$. Clearly we have
$$
A_r|f|(x)=a_k \ \text{for all}\ B_r(x)\subset I_k
$$
and
$$
A_r |f|(x)\leq a_k \ \text{if}\ B_r(x)\cap I_j=\emptyset\ \text{for all}\ j\neq k.
$$
On the other hand, if $B_r(x)\cap I_j \neq \emptyset$ for some $j\neq k$, we get that $B_r(x)\cap I_{k+1}\neq \emptyset$ or $B_r(x)\cap I_{k-1}\neq\emptyset$. Observe that, in the first case $r\geq N_{k+1}-N_k-L_k\geq 8N_k$, while, in the second case, $r\geq N_k-N_{k-1}-L_{k-1}\geq N_k-2N_{k-1}\geq \frac{4}{5}N_k$. Therefore
\begin{align*}
A_r|f|(x)\leq \frac{1}{\mu(B_r(x))}\|f\|_{L^1(X)} = \frac{1}{C_dr^d}\leq \frac{1}{C_d}\left(\frac{5}{4N_k}\right)^d\leq a_k,    
\end{align*}
where the last inequality follows from \eqref{eq:g-large} and \eqref{boundL/N}, since 
\begin{align*}
\frac{C_d}{2^k}\left(\frac{4}{5}\right)^d\geq \left[\frac{M}{g(N_k)^{\frac{1}{d}}}+1\right]^d-1\geq\frac{(N_k+L_k)^d-N^d_k}{N^d_k}= \frac{\mu(I_k)}{N^d_k}\frac{2^d}{C_d}\geq\frac{1}{a_k2^kN^d_k}.
\end{align*}
\end{proof}
Hence,
\begin{align*}
\frac{\mu(Z_{N_k+L_k,f})}{(N_k+L_k)^d/g(N_k+L_k)}&\geq g(N_k+L_k)\frac{\mu(I_k)}{(N_k+L_k)^d}\\
&=g(N_k+L_k)\frac{C_d}{2^d}\frac{(N_k+L_k)^d-N^d_k}{(N_k+L_k)^d}\\
&\geq g(N_k+L_k)\frac{C_d}{2^d}\frac{L^d_k}{(N_k+L_k)^d}\\
&\geq \frac{g(N_k+L_k)}{g(N_k)}\frac{C_d}{2^d}\frac{N^d_k}{(N_k+L_k)^d}\\
&\geq \frac{C_d}{2^d}\frac{1}{(1+\frac{L_k}{N_k})^d}\to \frac{C_d}{2^d}\ \text{as}\ k\to\infty.
\end{align*}

\end{proof}

\section{Size of the frequency function in $\ell^p(\mathbb{Z})$}\label{Sec:Proof3}

\begin{proof}[Proof of Theorem \ref{thrm:lp}]
Although it suffices to prove the theorem for $1<p<+\infty$, we begin with the case $p = +\infty$, since this case motivates the others. Let $N_1:= 2$, and for each positive integer $k$, we define
\begin{equation*}
    N_{k+1} := N_k^{10}, \quad L_k := \left \lfloor \frac{1}{3} N_k \right \rfloor, \quad \text{and}\quad I_k = [N_k+1, N_k+L_k].
\end{equation*}
We consider the function
\begin{equation*}
    f := \sum_{k=2}^{\infty} \mathds{1}_{I_k}.
\end{equation*}
Clearly $f \in \ell^{\infty}(\mathbb{Z})$. We shall prove that the set
\begin{equation*}
    \left\{ n \in \mathbb{Z} : \frac{r_n}{|n|} \in \left[ \frac{1}{4} , \frac{3}{4} \right] \right\}
\end{equation*}
has infinitely many elements. 

Note that if $k\geq 2$ and $1\leq r < L_k$, then
\begin{equation*}
    A_r|f|(N_k) = \frac{r}{2r+1} < \frac{L_k}{2L_k+1} = A_{L_k}|f|(N_k).
\end{equation*}
Since $L_k \to +\infty$ as $k \to +\infty$, there exists a positive integer $K$ such that
\begin{equation*}
    A_{L_k}|f|(N_k) = \frac{L_k}{2L_k+1} > \frac{5}{12}
\end{equation*}
for all $k\geq K$. Let $k\geq K$ be an arbitrary but fixed integer. If $N_k - N_{k-1} - L_{k-1} \leq r \leq N_k - N_{2}-1$, then
\begin{equation*}
\begin{aligned}
    A_r|f|(N_k) &\leq \frac{L_k + L_{k-1} + \dots + L_2}{N_k}\\ &\leq \frac{1}{3} \frac{N_k + N_{k-1} + \dots + N_2 }{N_k}\\ &< \frac{1}{3} \left( 1 + \sum_{j=9}^{\infty} \frac{1}{2^j} \right) \\ &<\frac{5}{12},
\end{aligned}
\end{equation*}
where in the first inequality we used the fact that $2r+1 > N_k$ (since $N_k>2N_{k-1}+2L_{k-1}$) and we also used that $N_2+1\leq N_k-r\leq N_k+r< N_{k+1}$, the second inequality follows from the definition of $L_k$, and the third inequality follows from the definition of the $N_k$'s as sparse powers of 2. 
Similarly, if $N_{k+j} - N_k + 1 \leq r \leq N_{k+j} - N_k + L_{k+j}$ for some positive integer $j$, then
\begin{equation*}
    \begin{aligned}
        A_r|f|(N_k) &\leq \frac{L_{k+j} + L_{k+j-1} + \dots + L_2}{N_{k+j}}\\ &\leq \frac{1}{3} \frac{N_{k+j} + N_{k+j-1} + \dots + N_2 }{N_{k+j}}\\ &< \frac{1}{3} \left( 1 + \sum_{i=9}^{\infty} \frac{1}{2^i} \right) \\ &<\frac{5}{12}.
    \end{aligned}
\end{equation*}
Observe that if $r$ does not satisfy any of the above inequalities, it follows that $N_k-r \not\in I_{j_1}$ and $N_k+r \not \in I_{j_2}$ for any positive integers $j_1$ and $j_2$, which implies the existence of $r' \in \mathbb{N}$ such that $A_r|f|(N_k) < A_{r'}|f|(N_k)$. Thus, $A_r |f|(N_k) < A_{L_k} |f|(N_k)$ for all $r \in \mathbb{N} \setminus \{L_k\}$. As a result, $r_{N_k} = L_k$ for every $k\geq K$. Given that 
\begin{equation*}
    \lim_{k\to \infty} \frac{r_{N_k}}{N_k} = \lim_{k\to \infty} \frac{L_k}{N_k} = \frac{1}{3},
\end{equation*}
we conclude that
\begin{equation*}
    N_k \in \left\{ n \in \mathbb{Z} : \frac{r_n}{|n|} \in \left[ \frac{1}{4} , \frac{3}{4} \right] \right\}
\end{equation*}
for infinitely many $k$'s. This concludes the proof of the case $p = +\infty$.

Now we discuss the case $p \in (1,+\infty)$. We choose an $\alpha \in (0,1)$ such that $\alpha p > 1$. Let $N_1 = 2^{\lceil 10/(1-\alpha) \rceil}$, and for each positive integer $k$ let us consider
\begin{equation*}
    N_{k+1} = N_k^{10}, \quad L_k = \left \lfloor \frac{1}{3} N_k \right \rfloor, \quad I_k = [N_k+1, N_k+L_k], \quad n_k = N_k + L_k + 1.
\end{equation*}
Let us consider the function
\begin{equation*}
    f(n) = \frac{1}{n^{\alpha}}\sum_{k=1}^{\infty} \mathds{1}_{I_k}(n),
\end{equation*}
which implies that $f \in \ell^{p}(\mathbb{Z})$. We shall prove that the set
\begin{equation*}
    \left\{ n \in \mathbb{Z} : \frac{r_n}{|n|} \in \left[ \frac{1}{8} , \frac{7}{8} \right] \right\}
\end{equation*}
has infinitely many elements.

Observe that if $k$ is a positive integer and $1\leq r < L_k$, then
\begin{equation*}
\begin{aligned}
    A_r |f|(n_k) &= \frac{1}{2r+1} \sum_{j=0}^{r-1} \frac{1}{(N_k+L_k-j)^{\alpha}} \\ &< \frac{1}{2L_k+1} \sum_{j=0}^{L_k-1} \frac{1}{(N_k+L_k-j)^{\alpha}}\\ &= A_{L_k} |f|(n_k).
\end{aligned}
\end{equation*}
In addition,
\begin{equation*}
    \begin{aligned}
        A_{L_k} |f|(n_k) &= \frac{1}{2L_k+1} \sum_{j=0}^{L_k-1} \frac{1}{(N_k+L_k-j)^{\alpha}}\\ &> \frac{1}{2L_k+1}\int_{N_k+1}^{N_k+L_k} \frac{1}{x^{\alpha}}dx\\ &= \frac{1}{1-\alpha} \frac{(N_k + L_k)^{1-\alpha} - (N_k + 1)^{1-\alpha}}{2L_k+1}.
    \end{aligned}
\end{equation*}
Given that $N_k(2L_k+1)^{-1} \to 3/2$ as $k \to \infty$, there exists a positive integer $K$ such that
\begin{equation*}
    A_{L_k} |f|(n_k) > \frac{7}{5} \frac{1}{1-\alpha} \left( \left( \frac{4}{3} \right)^{1-\alpha} - 1 \right) \frac{1}{N_k^{\alpha}}
\end{equation*}
for every $k\geq K$ (in fact, $7/5$ can be replaced by any constant $1<c<3/2$). Moreover, since $k/N_k^{1-\alpha} \to 0$ as $k \to \infty$, we may assume without loss of generality that
\begin{equation*}
    k < \frac{1}{100} \frac{1}{1-\alpha} \left( \left( \frac{4}{3} \right)^{1-\alpha} - 1 \right) N_k^{1-\alpha}
\end{equation*}
for all $k\geq K$. Let $k\geq K$ be an arbitrary but fixed integer. If 
$n_k - N_{k-1} - L_{k-1} \leq r \leq n_k - N_1 - 1$, then
\begin{equation*}
    \begin{aligned}
        A_r|f|(n_k) &\leq \frac{1}{N_k} \sum_{i=1}^{k}\sum_{j=1}^{L_i} \frac{1}{(N_i+j)^{\alpha}} \\ &\leq \frac{1}{N_k} \sum_{i=1}^{k} \left( 1 + \int_{N_i+1}^{N_i+L_i} \frac{1}{x^{\alpha}}dx \right)\\ &= \frac{k}{N_k} + \frac{1}{1-\alpha} \frac{1}{N_k} \sum_{i=1}^{k} [(N_i + L_i)^{1-\alpha} - (N_i + 1)^{1-\alpha}]\\ &\leq \frac{k}{N_k} + \frac{1}{1-\alpha} \left( \left( \frac{4}{3} \right)^{1-\alpha} - 1 \right)\frac{1}{N_k} \sum_{i=1}^{k} N_i^{1-\alpha}\\ &= \frac{k}{N_k} + \frac{1}{1-\alpha} \left( \left( \frac{4}{3} \right)^{1-\alpha} - 1 \right)\frac{1}{N_k^{\alpha}} \sum_{i=1}^{k} \left(\frac{N_i}{N_k}\right)^{1-\alpha}. 
    \end{aligned}
\end{equation*}
Because of the choice of $N_1$ and the other $N_i$'s, we have that $\sum_{i=1}^{k} \left(\frac{N_i}{N_k}\right)^{1-\alpha} < 1 + \sum_{j=9}^{\infty} \frac{1}{2^j}$, which implies 
\begin{equation*}
    A_r|f|(n_k) < \frac{7}{5} \frac{1}{1-\alpha} \left( \left( \frac{4}{3} \right)^{1-\alpha} - 1 \right) \frac{1}{N_k^{\alpha}} < A_{L_k}|f|(n_k).
\end{equation*}
Moreover, if $N_{k+m} - n_k + 1 \leq r \leq N_{k+m} - n_k + L_{k+m}$ for some positive integer $m$, then
\begin{equation*}
    \begin{aligned}
        A_r|f|(n_k) &\leq \frac{1}{N_{k+m}} \sum_{i=1}^{k+m}\sum_{j=1}^{L_i} \frac{1}{(N_i+j)^{\alpha}} \\ &\leq \frac{1}{N_{k+m}} \sum_{i=1}^{k+m} \left( 1 + \int_{N_i+1}^{N_i+L_i} \frac{1}{x^{\alpha}}dx \right)\\ &= \frac{k+m}{N_{k+m}} + \frac{1}{1-\alpha} \frac{1}{N_{k+m}} \sum_{i=1}^{k+m} [(N_i + L_i)^{1-\alpha} - (N_i + 1)^{1-\alpha}]\\ &\leq \frac{k+m}{N_{k+m}} + \frac{1}{1-\alpha} \left( \left( \frac{4}{3} \right)^{1-\alpha} - 1 \right)\frac{1}{N_{k+m}} \sum_{i=1}^{k+m} N_i^{1-\alpha}\\ &\leq \frac{k+m}{N_{k+m}} + \frac{1}{1-\alpha} \left( \left( \frac{4}{3} \right)^{1-\alpha} - 1 \right)\frac{1}{N_{k+m}^{\alpha}} \left( 1 + \sum_{j=9}^{\infty} \frac{1}{2^j} \right),
    \end{aligned}
\end{equation*}
and thus
\begin{equation*}
    A_r|f|(n_k) < \frac{7}{5} \frac{1}{1-\alpha} \left( \left( \frac{4}{3} \right)^{1-\alpha} - 1 \right) \frac{1}{N_{k+m}^{\alpha}} < A_{L_k}|f|(n_k).
\end{equation*}
As a result, $r_{n_k} = L_k$ for any $k\geq K$. Since
\begin{equation*}
    \lim_{k\to \infty} \frac{r_{n_k}}{n_k} = \lim_{k\to \infty} \frac{L_k}{N_k + L_k + 1} = \frac{1}{4},
\end{equation*}
we deduce that 
\begin{equation*}
    n_k \in \left\{ n \in \mathbb{Z} : \frac{r_n}{|n|} \in \left[ \frac{1}{8} , \frac{7}{8} \right] \right\}
\end{equation*}
for infinitely many $k$'s.
\end{proof}

\section{Appendix} \label{AppendixProof}

\subsection{Proof of the measurability of the uncentered frequency function}

As in the proof of \cite[Theorem 1.1]{TemurContinuous}, we establish the measurability of $\widetilde{r}(x) := \widetilde{r}_x$ by approximating it as the limit of a sequence of measurable functions. More precisely, we construct a sequence of measurable functions $\{\widetilde{r}_{k,l}\}_{k,l\geq 1}$ such that, for each fixed $l \in \mathbb{Z}^+$, the sequence $(\widetilde{r}_{k,l})_{k\geq 1}$ converges pointwise to a function $\widetilde{r}_l$. We then prove that $\widetilde{r}$ is the limit of the sequence $(\widetilde{r}_l)$ as $l \to +\infty$.

For any $k,l \in \mathbb{Z}^+$, let us consider the set
\begin{equation*}
    \widetilde{E}_{f,x,k,l} := \{r \in \mathbb{Q} \cap [2^{-l},2^l] :\text{there is } y \in \mathbb{Q}^d \text{ such that } x \in \overline{B_r(y)} \text{ and } A_r|f|(y) + 2^{-k} \geq \widetilde{M}f(x)\}.
\end{equation*}
Moreover, we define the function
\begin{equation*}
    \widetilde{r}_{k,l}(x) := \begin{cases}
        \inf \widetilde{E}_{f,x,k,l}, &\text{if } \widetilde{E}_{f,x,k,l} \neq \emptyset,\\
        0, &\text{otherwise}.
    \end{cases}
\end{equation*}

Let $k,l \in \mathbb{Z}^+$ and $\alpha \in \mathbb{R}$ be fixed. Let us prove that the set $\widetilde{r}_{k,l}^{-1}([\alpha,+\infty))$ is measurable. Notice that it is enough to consider the case $2^{-l} \leq \alpha \leq 2^{l}$, since $\widetilde{r}_{k,l}^{-1}([\alpha,+\infty))$ equals $\emptyset$, $\mathbb{R}$, or $\widetilde{r}_{k,l}^{-1}([2^{-l},+\infty))$ whenever $\alpha < 2^{-l}$ or $\alpha > 2^l$. For each $r>0$, let us consider 
\begin{equation*}
    \widetilde{S}_r := \{x \in \mathbb{R}^d : \text{there is } y \in \mathbb{Q}^d \text{ such that } x \in \overline{B_r(y)} \text{ and } A_r|f|(y) + 2^{-k} \geq \widetilde{M}f(x)\}.
\end{equation*}
Observe that each one of these sets can be written as
\begin{equation*}
    \widetilde{S}_r = \bigcup_{y \in \mathbb{Q}^d} \{x \in \overline{B_r(y)} : A_r|f|(y) + 2^{-k} \geq \widetilde{M}f(x)\},
\end{equation*}  
which implies that each $\widetilde{S}_r$ is measurable. Since
\begin{equation*}
    \widetilde{r}_{k,l}^{-1}([\alpha,+\infty)) = \left(\bigcup_{r \in [\alpha,2^l] \cap \mathbb{Q}}\widetilde{S}_r\right) \setminus \left(\bigcup_{r \in [2^{-l},\alpha) \cap \mathbb{Q}}\widetilde{S}_r\right),
\end{equation*}
it follows that $\widetilde{r}_{k,l}^{-1}([\alpha,+\infty))$ is measurable. Hence, each function $\widetilde{r}_{k,l}$ is measurable.

Next, notice that 
\begin{equation*}
    \widetilde{E}_{f,x,1,l} \supseteq \widetilde{E}_{f,x,2,l}\supseteq \widetilde{E}_{f,x,3,l}\supseteq \dots.
\end{equation*}
If one of the sets in this chain is empty, then $\widetilde{r}_{k,l}(x) = 0$ for all $k$ sufficiently large. If there is no empty set in the chain, then
\begin{equation*}
    \widetilde{r}_{1,l}(x) \leq \widetilde{r}_{2,l}(x) \leq \widetilde{r}_{3,l}(x) \leq \dots.
\end{equation*}
Since $\widetilde{r}_{k,l}(x) \leq 2^l$ for all $k \in \mathbb{Z}^+$, we have that the sequence $(\widetilde{r}_{k,l}(x))_{k\geq 1}$ converges for all $l \in \mathbb{Z}^+$. As a result, the function
\begin{equation*}
    \widetilde{r}_l(x) := \lim_{k\to \infty}\widetilde{r}_{k,l}(x)
\end{equation*}
is well-defined and is measurable.

Finally, we show that $\widetilde{r}_l(x) \to \widetilde{r}(x)$ as $l \to +\infty$ for every $x \in \mathbb{R}^d$. We shall consider three disjoint cases.
\begin{itemize}
    \item Let us assume that $\widetilde{E}_{f,x} = \emptyset$. Let $l \in \mathbb{Z}^+$ be fixed. Notice that 
    \begin{equation*}
        F = \{(y,r) \in \mathbb{R}^d \times \mathbb{R}^+ : r \in [2^{-l},2^l] \text{ and } x \in \overline{B_r(y)}\}
    \end{equation*}
    is a closed and bounded set. Thus, there exists $(y,r_y) \in F$ such that 
    \begin{equation*}
        \max_{(z,r) \in F} A_r|f|(z) = A_{r_y}|f|(y).
    \end{equation*}
    Since $\widetilde{E}_{f,x} = \emptyset$, it follows that $A_{r_y}|f|(y) < \widetilde{M}f(x)$. Let $K$ be a positive integer such that $2^{-K} < \widetilde{M}f(x) - A_{r_y}|f|(y)$. This implies that $\widetilde{E}_{f,x,k,l} = \emptyset$ for all $k \geq K$. As a result, $\widetilde{r}_l(x) = 0 = \widetilde{r}(x)$ for every $l \in \mathbb{Z}^+$.

    \item Let us assume that $\widetilde{r}(x)>0$. Hence, there exists $y \in \mathbb{R}^d$ such that $x \in \overline{B_{\widetilde{r}(x)}(y)}$ and $A_{\widetilde{r}(x)}|f|(y) = \widetilde{M}f(x)$. Let $l \in \mathbb{Z}^+$ be such that $2^{-l} < \widetilde{r}(x)/2 < 2^{l-1}$. Because of the continuity of the average function, for each $k \in \mathbb{Z}^+$ there exists $\alpha_k < \widetilde{r}(x)$ such that $\alpha_k \in \widetilde{E}_{f,x,k,l}$. This implies that $\widetilde{r}_{k,l}(x) \leq \widetilde{r}(x)$ for all $k \in \mathbb{Z}^+$, and thus $\widetilde{r}_{l}(x) \leq \widetilde{r}(x)$. On the other hand, since $\widetilde{E}_{f,x,k,l} \neq \emptyset$, an application of Bolzano-Weierstrass theorem yields that for each $k \in \mathbb{Z}^+$ there exists $y_k \in \mathbb{R}^d$ such that $x \in \overline{B_{\widetilde{r}_{k,l}(x)}(y_k)}$ and $A_{\widetilde{r}_{k,l}(x)}|f|(y_k) + 2^{-k} \geq \widetilde{M}f(x)$. Applying the Bolzano-Weierstrass theorem once again, there exists $y \in \mathbb{R}^d$ such that $y_k \to y$ as $k\to+\infty$ (possibly passing to a subsequence). Letting $k \to +\infty$, we get $\widetilde{r}_l(x) \in \widetilde{E}_{f,x}$. Hence, $\widetilde{r}_{l}(x) = \widetilde{r}(x)$ for all $l$ sufficiently large.

    \item Let us assume that $\widetilde{r}(x)=0$ and $\widetilde{E}_{f,x} \neq \emptyset$. This means that there exists a sequence $(\rho_n)$ in $\widetilde{E}_{f,x}$ such that $\rho_n \to 0$. We know that for each $n$ there exists $y_n \in \mathbb{R}^d$ such that $x \in \overline{B_{\rho_n}(y_n)}$ and $A_{\rho_n}|f|(y_n) = \widetilde{M}f(x)$. Let $n \in \mathbb{Z}^+$ be fixed, and let $l \in \mathbb{Z}^+$ be such that $2^{-l} < \rho_n/2 < 2^{l-1}$. Applying the continuity of the average operator, we deduce that for each $k$ there exists $\beta_k \in \widetilde{E}_{f,x,k,l}$ such that $\beta_k < \rho_n$. It follows that $\widetilde{r}_{k,l}(x) \leq \rho_n$ for all $k \in \mathbb{Z}^+$, which implies that $\widetilde{r}_{l}(x) \leq \rho_n$ for all $l$ sufficiently large. Therefore, $\limsup_{l\to+\infty}\widetilde{r}_{l}(x) \leq \rho_n$ for all $n \in \mathbb{Z}^+$. Given that $\rho_n\to 0$, we conclude that $\widetilde{r}_l(x) \to 0$ as $l \to +\infty$.
\end{itemize}

\subsection{Proof Sketch of Proposition \ref{prop:Prop1}}

    We start discussing the proof of \eqref{eq:Prop1}. Let $C>1$. For the sake of contradiction, let us assume that \eqref{eq:Prop1} does not hold. Thus, there exist $\varepsilon>0$ and a strictly increasing sequence of positive integers $(N_k)$ such that
    \begin{equation*}
        \frac{\mu(S_{N_k,f})}{N_k^d} \geq \varepsilon
    \end{equation*}
    for all $k \in \mathbb{Z}^+$. We assume without loss of generality that $S_{N_k,f} \setminus S_{N_{k-1},f}$ has positive Lebesgue measure for all $k\geq 2$. For a fixed $x \in S_{N_k,f} \setminus S_{N_{k-1},f}$, there exists $\alpha_x>0$ such that $\alpha_x \leq |x|/C$ and $A_{\alpha_x}|f|(x) \geq Mf(x)/2$. We can assume without loss of generality that $N_1$ is sufficiently large such that
    \begin{equation*}
        Mf(x)\geq A_{2|x|}|f|(x)\geq \frac{1}{C_d2^d|x|^d}\int_{B(0,N_1)}|f(y)|\, dy \geq \frac{\|f\|_1}{C_d2^{d+1}|x|^d}
    \end{equation*}
    for all $x\in S_{N_k,f} \setminus S_{N_{k-1},f}$ with $k\geq 2$, which implies
    \begin{equation} \label{ineqVitali}
        \frac{\|f\|_1}{C_d2^{d+2}|x|^d} \leq \frac{1}{C_d \alpha_x^d} \int_{B_{\alpha_x}(x)}|f(y)| \, dy.
    \end{equation}
    Observe that
    \begin{equation*}
        S_{N_k,f} \setminus S_{N_{k-1},f} \subseteq \bigcup_{x \in S_{N_k,f} \setminus S_{N_{k-1},f}} B_{\alpha_x}(x).
    \end{equation*}
    Moreover, since $\alpha_x \leq |x|/C$, by the triangle inequality, the union of these balls is contained in the annulus
    \begin{equation*}
        U_k = \left\{y \in \mathbb{R}^d : \left( 1 - \frac{1}{C} \right)N_{k-1} \leq |y| \leq \left( 1 + \frac{1}{C} \right)N_k\right\}.
    \end{equation*}
    Hence, we can choose suitable sparse conditions for the sequence $(N_k)$ to guarantee that the annuli $U_{2k}$ are pairwise disjoint. We use the inner regularity of the Lebesgue measure to approximate $S_{N_k,f} \setminus S_{N_{k-1},f}$ by compact sets, for a given compact set $K$ there is a finite union of balls $B_{\alpha_x}(x)$ whose union covers $K$, and we can apply Vitali's covering lemma to find a disjoint sub-collection of these balls such that the union of the dilation of these balls by a factor 3 contains $K$. Combining this with \eqref{ineqVitali} as in the proof of \cite[Theorem 1.3]{TemurContinuous} we obtain a contradiction.\smallskip
    
    Now we discuss the proof of \eqref{LimUncentered} when $X = \mathbb{Z}$. The case $X = \mathbb{R}^d$ can be deduced by applying analogous ideas. Let $C>2$ and $\widetilde{S}_{N,f}^+ = \widetilde{S}_{N,f} \cap \mathbb{Z}^+$. We shall only show that
    \begin{equation}\label{LimUncentered1}
        \lim_{N\to \infty} \frac{\nu(\widetilde{S}_{N,f}^+)}{N} = 0,
    \end{equation}
    as the proof of the corresponding limit for the sets $\widetilde{S}_{N,f}^- = \widetilde{S}_{N,f} \cap \mathbb{Z}^-$ is completely analogous. We proceed via contradiction, let us assume that \eqref{LimUncentered1} does not hold. Hence, there exist $\varepsilon>0$ and a strictly increasing sequence of positive integers $(N_k)$ such that
    \begin{equation*}
        \frac{\nu(\widetilde{S}_{N_k,f}^+)}{N_k} \geq \varepsilon
    \end{equation*}
    for all $k \in \mathbb{Z}^+$. We assume without loss of generality that $\widetilde{S}_{N_k,f}^+ \setminus \widetilde{S}_{N_{k-1},f}^+\neq\emptyset$. For a fixed $n \in \widetilde{S}_{N_k,f}^+ \setminus \widetilde{S}_{N_{k-1},f}^+$, we know that there is a pair of non-negative integers $(\rho_n,s_n)$ such that
    \begin{equation*}
        \widetilde{M}f(n) = \frac{1}{\rho_n + s_n + 1} \sum_{j=-\rho_n}^{s_n} |f(n+j)| \quad \text{and} \quad \widetilde{r}_n = \frac{\rho_n + s_n}{2}.
    \end{equation*}
    Since $2\widetilde{r}_n = \rho_n + s_n \geq \max\{\rho_n,s_n\}$, we have
    \begin{equation} \label{ineqVitali2}
        \widetilde{M}f(n) = \frac{1}{\rho_n + s_n + 1} \sum_{j=-\rho_n}^{s_n} |f(n+j)| \leq \frac{1}{2\widetilde{r}_n + 1} \sum_{j=-2\widetilde{r}_n}^{2\widetilde{r}_n} |f(n+j)|.
    \end{equation}
    Observe that 
    \begin{equation*}
        \widetilde{S}_{N_k,f}^+ \setminus \widetilde{S}_{N_{k-1},f}^+ \subseteq \bigcup_{n \in \widetilde{S}_{N_k,f}^+ \setminus \widetilde{S}_{N_{k-1},f}^+} [n-2\widetilde{r}_n , n+2\widetilde{r}_n] \cap \mathbb{Z}.
    \end{equation*}
  Since each interval in this covering is contained in $[(1-2/C)N_{k-1} , (1+2/C)N_k]$, we get a contradiction after imposing suitable conditions in the sequence $(N_k)$ and using Vitali's covering lemma and \eqref{ineqVitali2} as in the proof of \cite[Theorem 2]{TemurDiscrete}.

\section{Acknowledgements}
J.M. was partially supported by the AMS Stefan Bergman Fellowship and the Simons Foundation MPS-TSM
00007545. 

\bibliographystyle{acm}

\bibliography{bibl}

\end{document}